\documentclass[12pt, a4paper]{amsart}
\usepackage{amsmath,amssymb,amscd,amsfonts}

\usepackage{ifthen}
\usepackage[T1]{fontenc}
\usepackage[utf8]{inputenc}
\usepackage[all]{xy}
\usepackage{graphicx}
\usepackage{enumerate}
\usepackage{xspace}
\usepackage{pdfsync}
\usepackage{epic}
\usepackage{dsfont}
\usepackage{tikz-cd}
\usepackage{hyperref}

\def \Q{{\mathbb Q}}
\def \P{{\mathbb P}}
\def \Z{{\mathbb Z}}
\def \C{{\mathbb C}}
\def \R{{\mathbb R}}

\newtheorem{theorem}{Theorem}
\newtheorem{remark}[theorem]{Remark}
\newtheorem{definition}[theorem]{Definition}
\newtheorem{lemma}[theorem]{Lemma}

\newtheorem{proposition}[theorem]{Proposition}
\newtheorem{conjecture}[theorem]{Conjecture}
\newtheorem{corollary}[theorem]{Corollary}
\newtheorem{example}[theorem]{Example}
\newtheorem{question}[theorem]{Question}

\DeclareMathOperator{\Gr}{Gr}

\begin{document}
\title[]{The nilpotent quotients of normal quasi-projective varieties with proper quasi-Albanese map}

\dedicatory{\`A la mémoire de Jean-Pierre Demailly,\\
 pour ses contributions mathématiques fondamentales,\\
 et son inlassable générosité.}

\author{Rodolfo Aguilar Aguilar}
\address{\parbox{\linewidth}{University of Miami, Coral Gables, USA and Institute for Mathematics and Informatics, Bulgarian Academy of Sciences, Sofia, Bulgaria.}}
\thanks{The first author is partially supported by the Bulgarian Ministry of Education and Science, Scientific Programme "Enhancing the Research Capacity in Mathematical Sciences (PIKOM)", No. DO1-67/05.05.2022.}

\email{\href{mailto:aaguilar.rodolfo@gmail.com}{aaguilar.rodolfo@gmail.com}}
\urladdr{\url{https://sites.google.com/view/rodolfo-aguilar/}}

\author{Fr\'ed\'eric Campana}
\address{Universit\'e Lorraine \\
 Institut Elie Cartan\\
Nancy, France}

\email{frederic.campana@univ-lorraine.fr}

\date{\today}

\begin{abstract} We show that if $X$ is a normal complex quasi-projective variety, the quasi-Albanese map of which is proper, then the torsionfree nilpotent quotients of $\pi_1(X)$ are, up to a controlled finite index, the same ones as those of the normalisation of its quasi-Albanese image. When $X$ is quasi-K\"ahler smooth, we get the same conclusion, but only for the smooth models  of the quasi-Albanese image. In this second case, the proof is elementary, as the one given in \cite{Ca} for $X$ compact.

In the normal quasi-projective case, the \'etale Galois cover of $X$ associated to the nilpotent completion of $\pi_1(X)$ is thus holomorphically convex. This is proved in the smooth case by $3$ other methods in \cite{GGK}, which motivated the present text.

When  $X$ is `special' in the sense of \cite{Ca11}, we deduce that the torsion free nilpotent quotients of $\pi_1(X)$ are abelian. Examples show that this property fails (as first observed in \cite{CDY}) when the quasi-Albanese map is not proper. This leads to replace our previous `Abelianity conjecture' in the compact case by an `Nilpotency conjecture' in the non-compact quasi-K\"ahler context. 

\end{abstract}

\maketitle

\tableofcontents

\section{Introduction}

A quasi-K\"ahler manifold $X$ is a Zariski open subset of a compact K\"ahler manifold $\overline{X}$. A `good compactification' $(\overline{X},D)$ of $X$ is a smooth compact K\"ahler manifold $\overline{X}$ together with a simple normal crossings divisor $D$ of $\overline{X}$ such that $\overline{X}\setminus D=X$.

We denote with $\alpha_X:X\to Alb(X)$ the quasi-Albanese map of $X$, introduced by S. Iitaka (see \cite{I}), and with $Z_0''\subset Alb(X)$ its image. We denote with $Z'_0\to Z_0''$ the normalisation of $Z_0''$, and still write $\alpha_X:X\to Z'_0$ for the lift of $\alpha_X$ to $Z'_0$.

\medskip

{\bf Main assumption:} As in \cite{GGK}, we always assume here that $\alpha_X:X\to Alb(X)$ is {\bf proper}, so that $Z_0''\subset Alb(X)$ is closed and quasi-K\"ahler.

\medskip

By modifying $X$, which does not change $\pi_1(X)$, we get the following commutative diagram, where $\sigma':=d\circ \alpha$ is the Stein factorisation of $\alpha_X=\varphi'\circ \sigma'$, so that $Z'$ is normal, the fibres of $\sigma':X\to Z'$ are connected, and those of $\varphi':Z'\to Z'_0$ are finite. Moreover, $d:Z\to Z'$ and $d_0:Z_0\to Z'_0$ are desingularisations. The maps $\varphi$ and $\varphi'$ are thus finite, of the same degree $e$.

$$ \begin{matrix}
\xymatrix{X\ar[r]^{\alpha}\ar[rd]_{\alpha_0}&Z\ar[r]^{d}\ar[d]^{\varphi}& Z'\ar[d]^{\varphi'}\\
&Z_0\ar[r]^{d_0}&Z'_0\\
}
\end{matrix}$$

We may assume that all these maps extend holomorphically to suitable compatible compactifications $\overline{Z},\overline{Z'},\overline{Z_0},\overline{Z'_0}$ of $Z,Z',Z_0, Z'_0$, which are good for $Z,Z_0$.  

\medskip

We use the notations, and refer to the appendix of \cite{Ca} (see also \cite{S},\S7). In particular if $G$ is a group, and $G_n$ is the (nilpotent) quotient of $G$ by the $(n+1)-th$ term of its central series, and $G'_n:=G_n/Torsion$, with Kernel $K'_n(G)$, then we define $G^{nilp}:=G/\cap K'_n(G)$. These constructions are functorial.

We have the following elementary property (part of \cite{S}, Theorem 7.3): if $f:G\to H$ is a morphism of groups such that , for any $n\geq 1$, the induced morphism  $(K'_n(G)/K'_{n+1}(G))\otimes \Bbb Q\to (K'_n(H)/K'_{n+1}(H))\otimes \Bbb Q$ is isomorphic, then $G^{nilp}\to H^{nilp}$ is injective. From this, it is easy to deduce that if, moreover, $f:G\to H$ has an image of finite index $e$, then $G^{nilp}\to H^{nilp}$ has moreover an image of finite index at most $e$. In particular, if $f$ is surjective, then $G^{nilp}\to H^{nilp}$ is isomorphic.

\medskip

By $\pi^X_{nilp}:X^{nilp}\to X$, we denote the Galois \'etale cover of group $\pi_1(X)^{nilp}$. We use a similar notation for $Z,Z', Z_0,Z'_0$. If $f:V\to W$ is a holomorphic map, we denote by $f_{nilp}:V^{nilp}\to W^{nilp}$ the induced map.

\begin{remark}\label{finite index} The property of $G^{nilp}$ stated above, shows that if $f:V\to W$ is a surjective proper map with finite fibres, of degree $e$, between connected normal complex spaces, and if $(K'_n(G)/K'_{n+1}(G))\otimes \Bbb Q\to (K'_n(H)/K'_{n+1}(H))\otimes \Bbb Q$ is isomorphic for all $n>0$, with $G:=\pi_1(V), H:=\pi_1(W)$, then the induced map $G^{nilp}\to H^{nilp}$ is injective with image of finite index at most $e$. This follows from the fact that $f_*:G\to H$ has image of finite index at most $e$  (see \cite{Ca91}, Proposition 1.3). We shall repeatedly apply this remark to $\varphi:Z\to Z_0$ and $\varphi':Z'\to Z'_0$, and thus show the results stated for either  $Z$, or $Z'$ only, since the proofs apply to both cases.
\end{remark}

Our first result is:

\begin{theorem}\label{nilp smooth} If $X$ is quasi-K\"ahler  and $\alpha_X$ proper, then:
$\alpha_*:\pi_1(X)^{nilp}\to \pi_1(Z)^{nilp}$ is isomorphic, 
$(\alpha_0)_*: \pi_1(X)^{nilp}\to \pi_1(Z_0)^{nilp}$ is injective with image of finite index at most $e$.

Hence: $X^{nilp}=X\times_ZZ^{nilp}$, and $Z^{nilp}$ is a component of $Z\times_{Z_0}Z_0^{nilp}$.
\end{theorem}

\begin{corollary}\label{cor nilp smooth} In the situation of theorem \ref{nilp smooth}, we have:

1. $X^{nilp}$ is holomorphically convex if $Z_0^{nilp}$ is Stein.

2. If $\alpha_X^*:H^2(Alb(X),\Bbb Q)\to H^2(Z_0,\mathbb Q)$ is injective, $\pi_1(X)^{nilp}=H_1(X,\mathbb Z)/Torsion$.
This occurs, in particular, if $\alpha_X$ is surjective.
\end{corollary}

\begin{theorem}\label{nilp} If $X$ quasi-projective,and if $\alpha_X$ is proper:
$(d\circ \alpha)_*:\pi_1(X)^{nilp}\to \pi_1(Z')^{nilp}$ is an isomorphism, and $\varphi_*:\pi_1(Z')^{nilp}\to \pi_1(Z_0')^{nilp}$ is injective with image of finite index at most $e$.
\end{theorem}

The Claim 2 of the following Corollary is proved by 3 other methods in \cite{GGK}. It was previously proved for $Z'$ smooth in \cite{Ca}, remarque 1.9, and in \cite{K} for $X$ projective. 

\begin{corollary}\label{cor nilp} (\cite{GGK}) In the situation of theorem \ref{nilp}, we have:

1.  The connected components of the fibres of $\pi_1(X)^{nilp}\to \pi_1(Z')^{nilp}$ are compact.

2. $Z'^{nilp}$ is Stein, and $X^{nilp}$ is thus holomorphically convex.

\end{corollary}

The Steinness of $Z'^{nilp}$ is clear, since it is an \'etale cover of $Z'\times_A\tilde{A}$, $\tilde{A}$ being the (Stein) universal cover of $A$.

The proof of Theorem \ref{nilp} shows that its conclusion holds, more generally, for $X$ normal, instead of smooth. See Corollary \ref{cor alb} for the definition and properties of the quasi-Albanese map in the normal case.

\begin{theorem}\label{qp normal}
If $X$ is quasi-projective, normal and if $\alpha_X$ is proper, the induced map $\pi_1(X)^{nilp}\to \pi_1(Z')^{nilp}$ is an isomorphism, and $\pi_1(X)^{nilp}\to \pi_1(Z_0')^{nilp}$ is injective with image of finite index at most $e$.
\end{theorem}
Here, $Z_0'$ is, as before, the normalisation of $\alpha_X(X)$, and $Z'$ is the Stein factorisation of $X\to Z_0'$.  The Corollary \ref{cor nilp} thus remains valid for $X$ normal.

\medskip

Although the preceding theorems are essentially particular cases of Theorem \ref{qp normal}, it is instructive to give their proofs, which relies on simpler notions and techniques.

While the proof of Theorem \ref{nilp smooth} is elementary, directly extending to the open case the arguments of \cite{Ca}, the proof of Theorem \ref{nilp} relies on the mixed Hodge structures constructed in \cite{D} and \cite{H} on the cohomology and homotopy groups of complex algebraic varieties, which permit to deal with the, possibly arbitrary, singularities occuring in $Z_0''$. The proof of Theorem \ref{qp normal} is essentially the same one as of Theorem \ref{nilp}, once the Albanese map is defined in the normal context. The arguments should work in the quasi-K\"ahler case as well, but the corresponding statements have not been written in this context, as far as we know.

Corollary \ref{cor nilp smooth} applies to `special' quasi-K\"ahler manifolds as defined in \cite{Ca11} when the quasi-Albanese map is proper, and shows then that $\pi_1(X)^{nilp}=H_1(X,\mathbb Z)/Torsion$. We show however by examples in the last section that this fails in general if $\alpha_X$ is not proper, as first observed in \cite{CDY}. As a consequence, the `Abelianity conjecture' in the compact case needs to be weakened to the `Nilpotency conjecture' when $X$ is not compact.

Examples given in Remark \ref{counter} show that the holomorphic convexity of the universal cover of the smooth quasi-projective varieties commonly fails the open case.

\subsection{Acknowledgements}
We thank Robert Laterveer for the example of Hilbert surfaces given in Remark \ref{r1}, Morihiko Saito for the examples in Remark \ref{rMS}, the authors of \cite{CDY} for sending us their preprint,  and R. Hain for communicating his construction of the higher-Albanese in the normal quasi-projective case (we actually use only the case $s=1$, also constructed in \cite{CDY}).

\section{Proof of Theorem \ref{nilp smooth}}

When $X$ is compact K\"ahler, this is proved in \cite{Ca}, Corollaire 3.1. From the proof given there, it is sufficient to show that $\alpha^*:H^2(Z,\mathbb Q)\to H^2(X,\mathbb Q)$ is injective, which follows from the more general next Lemma:

\begin{lemma}\label{inj} Let $f:X\to Z$ be a proper surjective map between quasi-K\"ahler (connected) manifolds. 

For every $k\geq 0$, the map $f^*:H^k(Z,\mathbb Q)\to H^k(X,\mathbb Q)$ is injective.
\end{lemma}

\begin{proof} This is the same one as for the Lemme\footnote{The hypothesis that $X$ is compact has been forgotten in its statement.} 2.4 of \cite{Ca}, up to the non-compact version of Poincar\'e duality in cohomology. With the notations of the proof of this Lemma 2.4, it is sufficient to show that if $u\in H^k(Z,\Bbb Q)\neq 0$, there is $v\in H_c^{2m-k}(Z,\mathbb Q)$ such that $[u\wedge v]\neq 0\in H_c^{2m}(Z,\mathbb Q), m:=dim_{\mathbb C}(Z)$. By \cite{BT},Chap 1,\S 5, (especially remark 5.7), this is true if $Z$ has either a good finite cover, or more generally, if its cohomologies with compact support are finite dimensional. But this latter property is satisfied by quasi-K\"ahler manifolds, by the exact cohomology sequence with compact supports for a pair $(A,B)$ with $A\supset B$ compact, applied to a good compactification of $Z$: $A=\overline{Z}, B=D$, $Z=\overline{Z}\setminus D$.
\end{proof}

The consequences stated in Corollary \ref{cor nilp smooth} are then obvious.

We give in Corollary \ref{Albdom} situations in which $\alpha_X$ is dominant, hence surjective if proper.

\begin{remark}\label{r1} 1.The K\"ahler hypothesis of Lemma \ref{inj} is required already when $X$ is compact as the examples of Hopf surfaces of algebraic dimension $1$ and of the Iwazawa threefold show, see 2.5 and 2.6 of \cite{Ca}. The example \ref{ex} shows that the conclusion of lemma \ref{inj} may fail if the properness hypothesis is dropped, even if $Z$ is an elliptic curve and $a_{X}:X\to Z$ a $\mathbb C^*$-bundle.

2. The conclusion of Lemma \ref{inj} may also fail if $Z$, projective, is only assumed to be normal. There indeed exist compactified normal Hilbert modular surfaces $Z$ with elliptic singularities cusps such that the kernel of $H^2(Z)\to H^2(X)$ is non-zero, if $X$ is their minimal resolution (see \cite{BPV},Chap. V, \S22 ). We thank Robert Laterveer for indicating us this class of examples. These examples show that Stallings criterion (surjectivity of $H_2(G)\to H_2(Q)$ implies $G^{nilp}=Q^{nilp}$ for two groups with the same $H_1$ with $\mathbb Q$ coefficients if the group morphism $G\to H$ is onto) is a sufficient, but not necessary condition. See Theorem \ref{normal compact} below, which shows that the weaker condition that $H_2(G)$ and $H_2(H)$  have the same image in $H_2(G^{ab}/Torsion)$ is sufficient, in the quasi K\"ahlerian context, to get the same conclusion, where $G^{ab}$ denotes the abelianisation of $G$.
\end{remark}

\section{Proof of Theorem \ref{nilp}}

We start with proving the easier compact case.

\begin{theorem}\label{normal compact} Let $X$ be a projective manifold, $\alpha_X:X\to Alb(X)$ its Albanese map, and $\sigma:X\to Z'$ its Stein factorisation (so that $Z'$ is normal). Then $\pi_1(X)^{nilp}=\pi_1(Z')^{nilp}$.\end{theorem}

\begin{proof}  We can then apply \cite{ADH}, Theorem 1.2 and its corollary 1.3, which says that $\pi_1(Z')^{nilp}$ is $1$-formal. Since $H^1(X)=H^1(Z')$ (coefficients $\mathbb Q$), it is sufficient to show that $Ker(a_{Z'})^*:H^2(A)\to H^2(Z')=Ker(a_X^*)=Ker(a_{Z'}\circ \sigma)^*:H^2(A)\to H^2(X)$.This is a consequence of \cite{D}, 8.2.7.
\end{proof}

We now consider the open case. Recall the statement:

\begin{theorem}\label{normal} Let $X$ be a quasi-projective manifold, $\alpha_X:X\to Alb(X)$ its (proper) Albanese map, and $\sigma':X\to Z'$ its Stein factorisation (so that $Z'$ is normal). Then $\pi_1(X)^{nilp}=\pi_1(Z')^{nilp}$.
\end{theorem}

\begin{proof} By \cite{ADH}, Lemma 4.1, dualized, and the fonctoriality of the description of the Mal\u cev Lie algebra given there, it is sufficient to show that $Ker(H^2(A)\to H^2(Z'))=Ker(H^2(A)\to H^2(Z))$, with $A:=Alb(X)$, and $H^2(A)=\wedge^2H^1(A)$, since $A$ is a (commutative, complex Lie) group. By strictness, it is sufficient to check this for the weight graduations.

We shall show this on the graded quotients for the weights $2,3,4$ of the weight filtrations, since this is checked for the weights $0,1$ in \cite{ADH}.

\medskip

$\bullet$ {\bf Weights 3,4.} For the weights $3,4$, this follows from the more general Proposition \ref{injw>2}, and does not use the properness of $\alpha_X$. The weight $2$, which depends on this properness, will be treated after this.

The proof for the weights $3$ and $4$  relies on the statements and proofs of  [PS, 5.37 and 5.53], using the notations there.

\begin{lemma} Let $g:\tilde{Y}\to Y$ be a resolution of singularities of a quasi-projective variety $Y$. Let $D\subset Y$ be the discriminant locus and $E=g^{-1}(D)$, then the induced maps $\Gr_i^W H^2(D)\to \Gr_i^W H^2(E)$ are injective for $i=3,4$.
\end{lemma}
\begin{proof}
We proceed by induction on the dimension of $Y$. The cases for $Y$ of dimension $0$ or $1$ are trivial.

 The Mayer-Vietoris sequence for the discriminant square of [PS, Corollary-Definition 5.37], and strictness of MHS morphisms give:
$$\Gr_i^W H^k(Y)\to \Gr_i^W H^k(\tilde{Y})\oplus \Gr_i^W H^k(D)\to \Gr_i^W H^k(E)\to \Gr_i^W H^{k+1}(Y) $$

First suppose that $D$ is of dimension $0$. As $g$ is a proper morphism, we have: $\Gr_i^W H^k(D)=0$ and $\Gr_i^W H^k(E)=0$ for $i>k$ and $k\geq 0$. The result follows for the zero dimensional case.

Now suppose that $\dim Y=n$ and that the statement is true for quasi-projective varieties of dimension strictly less than $n$. Consider resolutions of singularities $\tilde{E}\to E$,  $\tilde{D}\to D$ and discriminant squares for the right-vertical maps of the two following diagrams: 

$$\begin{tikzcd}
G \arrow[r] \arrow[d] & \tilde{E} \arrow[d] \\
F \arrow[r] & E
\end{tikzcd}
and \begin{tikzcd}
G' \arrow[r] \arrow[d] & \tilde{D} \arrow[d] \\
F' \arrow[r] & D
\end{tikzcd}$$

The Mayer-Vietoris sequence [PS, Corollary-Definition 5.37] reads:

$$\Gr_i^W H^{1}(G)\to \Gr_i^W H^2(E)\to \Gr_i^W H^2(\tilde{E})\oplus \Gr_i^W H^2(F)\to \Gr_i^W H^2(G) $$

By induction hypothesis, we have that $\Gr_i^W H^2(F)\to \Gr_i^W H^2(G)$ is injective for $i=3,4$. As $H^1(G)$ has only weights $0,1,2$, we conclude by a simple diagram chasing using the definition of the arrows in this exact sequence, that $\Gr_i^W H^2(E)\to \Gr_i^W H^2(\tilde{E})$ is injective for $i=3,4$.

Similarly, we obtain that that $\Gr_i^W H^2(D)\to \Gr_i^W H^2(\tilde{D})$ is injective for $i=3,4$.

Now, note that the vertical lines in \begin{tikzcd}
\tilde{E} \arrow[r] \arrow[d] & E\arrow[d] \\
\tilde{D} \arrow[r] & D
\end{tikzcd} are proper maps. As $\tilde{E}$ and $\tilde{D}$ are smooth, we can use Lemma 2 together with the above horizontal and left-vertical injectivities to conclude that $\Gr_i^W H^2(D)\to \Gr_i^W H^2(E)$ is injective as well, for $i=3,4$.\end{proof}

\begin{proposition}\label{injw>2}Let $f:\tilde{Y}\to Y$ be a resolution of singularities of a quasi-projective variety, then the induced morphism $\Gr_i^W H^2(Y)\to \Gr_i^W H^2(\tilde{Y})$ is injective for $i=3,4$.
\end{proposition}
\begin{proof}
Consider again the Mayer-Vietoris sequence for the discriminant square \begin{tikzcd}
E \arrow[r] \arrow[d] & \tilde{Y} \arrow[d] \\
D \arrow[r] & Y
\end{tikzcd}, with the notation of the previous lemma. 

Note that, for $i=3,4$, $\Gr_i^WH^2 (Y)\to \Gr_i^2 H^2(\tilde{Y})$ is injective, if $\Gr_i^W H^2(D)\to \Gr_i^W(E)$ is injective.

Indeed, as $H^1(E)$ has only weights $0,1,2$, the map $\Gr_i^W H^2(Y)\to \Gr_i^W H^2(D)\oplus \Gr_i^W H^2(\tilde{Y})$ is injective for $i=3,4$. We conclude by the previous Lemma.\end{proof}

$\bullet$ {\bf Weight 2.} We still have to deal with the weight $2$ graded part of the cohomology of $A,Z',Z$. Let $\overline{A},\overline{Z'},\overline{Z}$ be compatible compactifications of $A,Z',Z$ respectively, the compactifications of $A$ and $Z$ being `good'.  Recall that $A:=Alb(X)$.

\medskip

We thus have a comutative diagram, in which the maps are the obvious ones:

$$ \begin{matrix}
\xymatrix{Z\ar[r]^{j}\ar[d]^{d}& \overline{Z}\ar[d]^{d'}\\
Z'\ar[r]^{j'}\ar[d]^{a}&\overline{Z'}\ar[d]^{a'}\\
A\ar[r]^{b}&\overline{A}\\
}
\end{matrix}$$

\medskip

This diagram induces the following commutative diagram:

$$ \begin{matrix}
\xymatrix{Gr^W_2H^2(Z)& H^2(\overline{Z})\ar[l]^{j^*}& K\ar[l]&0\ar[l]\\
Gr^W_2H^2(Z')\ar[u]^{d^*}&H^2(\overline{Z'})\ar[l]^{j'^*}\ar[u]^{d'^*}&K'\ar[u]^{k}\ar[l]&0\ar[l]\\
Gr^W_2H^2(A)\ar[u]^{a^*}&H^2(\overline{A})\ar[u]^{a'^*}\ar[l]^{b^*}&L\ar[l]\ar[u]&0\ar[l]\\
}
\end{matrix}$$

The kernels $K,K',L$ are naturally defined. We have replaced $Gr_2^WH^2(\bullet)$ by $H^2(\bullet)$ in the top and bottom lines of the central column. The map $j'^*$ is well-defined because $H^2(\overline{Z'})$ has weights $0,1,2$ only.

Moreover, the map $b^*$ (and $j^*$ as well) is surjective, by \cite{D'}, Corollaire 3.2.17.

We want to show that $Ker(d^*a^*)=Ker(a^*)$, or equivalently, since $b^*$ is surjective, that $Ker(d^*a^*b^*)=Ker(a^*b^*)$, which means that $Ker(j^*d'^*a'^*)=Ker(j'^*a'^*)$, or still equivalently, that $(d'^*a'^*)^{-1}(K)=(a'^*)^{-1}(K')$.

The Gysin maps show that $L$ (resp. $K$) is the complex subspace of $H^2(\overline{A})$ (resp. $H^2(\overline{Z}$)) generated by the Chern classes of the irreducible components of $D_A:=\overline{A}\setminus A$ (resp. of $D_Z:=\overline{Z}\setminus Z$). 

Let similarly $D_{Z'}:=\overline{Z'}\setminus Z'$. We obviously have that $K'$ contains $K"$, the complex vector subspace of $H^2(\overline{Z'})$ generated by the Chern classes of the components\footnote{More precisely: of sums of such components of the form $(a')^{-1}(D)$, for $D$ any component of $D_{\overline{A}}$, which are Cartier.}of $D_{Z'}=(a')^{-1}(D_A)$ . 

Indeed, since $a_Z:Z\to Alb(Z)=:A$ is proper, $D_Z=(a'd')^{-1}(D_A)$, and $D_{Z'}=(a')^{-1}(D_A)$. 

From \cite{D},Proposition 8.2.7, we get: $Ker(a'^*)=Ker((a'd')^*):=J$. 

Thus $(d'^*a'^*)^{-1}(K)=L+J$, and contains $(a'^*)^{-1}(K')$, which itself contains $L+J$, since $K'$ contains $K"$. But $L+J=((a'd')^*)^{-1}(K)$ contains $(a'^*)^{-1}(K')$, which contains $L+J$ and we get the claimed equality: $Ker(a'd'j)^*=Ker(a'd')^*+L=L+Ker(a'^*)=Ker(aj')^*$\end{proof}

The preceding proof shows that the injectivity statement for $Gr^W_2H^2$ holds more generally in the situation of the next Lemma \ref{Gr2}, which gives a partial extension of Deligne's result \cite{D}, Proposition 8.2.7, used in an essential way in the proof.

\begin{lemma}\label{Gr2} Let $f:X\to Y, g:Y\to A$ be two regular maps between quasi-projective varieties, such that $X,A$ is smooth, $f$ is surjective, and $g\circ f$ is proper. Then $Ker(g\circ f)^*:Gr^W_2H^2(A)\to Gr^W_2H^2(X))$ and $Ker(f)^*:Gr^W_2H^2(A)\to Gr^W_2H^2(Y))$ coincide.
\end{lemma}

\begin{remark} We do not know the answer to the following question. Let $f:X\to Y$ and $g:Y\to A$ be regular proper maps between quasi-projective varieties, such that $X$ and $A$ are smooth and $f$ is surjective. Does $Ker(g\circ f)*:Gr_i^W H^k(A)\to Gr_i^W H^k(X$) and $Ker(f)*: Gr_i^W H^k(A)\to Gr_i^W H^k(Y)$ coincide for all $i$?
\end{remark} 

\begin{remark}\label{rMS} The following examples, communicated to us by M. Saito \cite{M}, show that Proposition \ref{injw>2} does not necessarily hold for higher degrees: if $f:\tilde{Y}\to Y$ is a resolution of singularities of a quasi-projective variety, the induced morphisms $\Gr_i^W H^n(Y)\to \Gr_i^W H^n(\tilde{Y})$ do not need to be injective for $i=n, \ldots, 2n-2$ if $2n-2\geq n$. 
\begin{itemize}
\item 
Assume $X=X_1\times X_2$ with $X_1$ smooth and $X_2$ compact, e.g,
$X_1=\mathbb{C}^*$ and $X_2$ is the union of the zero section and a section of
a very positive line bundle on $\mathbb{P}^2$ whose zero locus $C$ is smooth.
Then $\tilde{X_2}=\mathbb{P}^2\cup \mathbb{P}^2$, but $\Gr^W_3 H^3(X)$ contains
$\Gr^W_2H^1(X_1)\times Gr^W_1H^2(X_2)$ and $Gr^W_1H^2(X_2)=H^1(C)(-1)$.

We can also have $X_2$ irreducible. Let $X=X_1\times X_2$ with $X_1=\mathbb{C}^*$ and $X_2$ the compactification of $f=f_1+f_2=0$ in
$\mathbb{C}^3$ with $f_1$,$f_2$ homogeneous polynomials of degree $d$ and $d+1$ having isolated
singularities at $0$. Here $\tilde{X_2}$ is the blow-up of $\mathbb{P}^2$ along the intersections of
the two curves defined by $f_1$, $f_2$, which are assumed to be transversal.
\item For the case $n$ is at most $2\dim X-2$, consider for instance $X=\mathbb{C}^*\times D\times A$
where $D$ is a plane curve with lots of nodes and $A$ is an abelian variety
\end{itemize}
\end{remark}

\section{Proof of Theorem \ref{qp normal}}

Before giving the proof, we need to define the quasi-Albanese map $\alpha_X:X\to Alb(X)$ for $X$ normal quasi-projective, inducing an isomorphism at $H^1$ level. This is done in the  Appendix on Weight-$1$ MHS (where it is denoted $\alpha_X:X\to A_X$). The crucial property used is: $H^1(X)=H^1(Alb(X))$ for coefficients either $\Q,\R$, or $\C$ (more precisely, this is true with $\Z$-coefficients, up to torsion).

\medskip

{\bf Proof of Theorem \ref{qp normal}:} We consider a smooth model $\tilde{X}$ of $X$ and a normal model $Z_0'$ of the image $Z_0''=\alpha_X(X)$. Consider the composed map $\tilde{X}\to Z'\to A:=Alb(X)$. This composed map is proper, since so are $\tilde{X}\to X$ and $X\to A$, by assumption. Moreover, $\tilde{X}$ and $A$ are smooth quasi-projective, and $\tilde{X}\to Z'$ is surjective. By Lemma \ref{Gr2}, the kernels of $\Gr^W_i H^2(A)\to \Gr^W_i H^2(\tilde{X})$ and of $\Gr^W_i H^2(A)\to \Gr^W_i H^2(Z')$ coincide, and so thus do the kernels of $\Gr^W_i H^2(A)\to \Gr^W_i H^2(X)$ and of $\Gr^W_i H^2(A)\to \Gr^W_i H^2(Z')$.  Moreover, the injectivity for $i=3,4$ of the maps $Gr^W_iH^2(X)\to Gr^W_iH^2(\tilde{X})$ is garanteed by Lemma \ref{injw>2}. We have the same property for a desingularisation $\tilde{Z'}\to Z'$, from which we deduce (as in the proof of Proposition \ref{injw>2}), the injectivity of $Gr^W_iH^2(Z')\to Gr^W_iH^2(X)$, and finally the equality of the kernels of $H^2(A)\to H^2(Z')$ and of $H^2(A)\to H^2(X)$, which gives the conclusion of Theorem \ref{qp normal}, since $dim H^1(X,\Bbb C)=dim(Alb(X))$.

\begin{remark} The kernel of the map $\pi_1(\widetilde{X})\to \pi_1(X)$ may be large, as shown by the cone $X$, either affine or projective, over a smooth projective manifold $E$, since $X$ is simply connected, while $\pi_1(\widetilde{X})\cong \pi_1(E)$. More generally:

\end{remark}

\begin{question} Let $d:\widetilde{X}\to X$ be the desingularisation of a normal quasi-projective variety, and $d_*:\pi_1(\widetilde{X})\to \pi_1(X)$ the induced morphism. Is the kernel $K$ of $d_*$ equal to the group $K'$ generated by the images in $\pi_1(X)$ of the $\pi_1(F)$, when $F$ runs through the fibres of $d$? It is clear that $K'\subset K$.

The second property stated in Proposition \ref{prop alb} shows that this is true for the torsionfree Abelianisations of the $\pi_1$'s, and Theorem \ref{qp normal} that it extends to  the torsionfree nilpotent completions when $\alpha_X$ is proper.
\end{question}

\section{Special Manifolds}

We show in this section that the Albanese map, if proper, is surjective with connected fibres if $X$ is a `special' quasi-projective manifold. For motivation and details about the class of special manifolds, see \cite{Ca04},\cite{Ca11}, and \cite{Ca16}. Recall first one of the possible definitions:

\begin{definition}\label{defspec} Let $X$ be a quasi-K\"ahler manifold together with a good compactification $X=\overline{X}\setminus D$.
We say that $X$ is special if, for any integer $p>0$ and any $L\subset \Omega^p_{\overline{X}}(Log(D))$, coherent and of rank $1$, we have: $\kappa(\overline{X}, L)<p$.
\end{definition}

This definition is short, but does not explicitely show the geometric significance of this notion: no surjective algebraic map $f:X\to Y$ has an `orbifold base' $(Y,\Delta_f)$ of Log-general type. Special manifold should be considered as higher-dimensional generalisations of quasi-projective curves with Log-canonical bundle of non-positive degree (ie: rational or elliptic curves in the compact case).

The importance of the notion of specialness comes from the existence, for any $X$, of its `core fibration' $c:X\to C(X)$, which splits $X$ into its antithetic parts: special (the general fibres), and Log-general type (the `orbifold base' $(C(X),\Delta_c)$.

Recall that the log-canonical bundle of $X$ is $K_{\overline{X}}+D=det(\Omega^1_{\overline{X}}(Log(D)))$, and that $\overline{\kappa}(X):=\kappa(\overline{X},K_{\overline{X}}+D)$. Also, $X$ is compact if and only if $D=0$, in which case $\overline{\kappa}(X)=\kappa(X)$.

\begin{proposition} Let $X$ be a quasi-K\"ahler manifold $X$.

1. If $X$ is special, and if $f:X\to Z$ is a dominant meromorphic map, then $\overline{\kappa}(Z)<dim(Z):=p>0$.

2. if $\overline{\kappa}(X)=0$, then $X$ is special.
\end{proposition}

\begin{proof} First claim:  Since $L:=f^*(K_{\overline{Z}}+D_Z))\subset \Omega^p_{\overline{X}}(Log(D))$, $f$ induces, for any $m>0$, an injective linear map $f^*:H^0(Z, m.(K_{\overline{Z}}+D_Z))\to H^0(X,Sym^m(\Omega^p_{\overline{X}}(Log(D)))$, where $(\overline{X},D)$ and $(\overline{Z},D_Z)$ are good compactifications of $X,Z$ respectively. We thus have: $p=\kappa(\overline{Z},K_{\overline{Z}}+D_Z)=\kappa(\overline{X}, f^*(K_{\overline{Z}}+D_Z))$, contradicting the specialness of $X$.

The second claim is \cite{Ca11}, Th\'eor\`eme 7.7.
\end{proof}

\begin{corollary}\label{cspec}Let $X$ be a special\footnote{The proof shows that the conclusion holds more generally if $X$ is `weakly special': no finite \'etale cover of $X$ maps onto a variety of Log-general type.} quasi-K\"ahler manifold. Assume that $\alpha_X:X\to Alb(X)$ is proper. Then every torsionfree nilpotent quotient of $\pi_1(X)$ is abelian, quotient of $H_1(X,\mathbb Z)/Torsion$. 
\end{corollary}

\begin{proof} This is an immediate consequence of Corollary \ref{cor nilp smooth}.2 and the next Lemma \ref{Albdom}.
\end{proof}

\begin{lemma}\label{Albdom} Let $\alpha_X:X\to Alb(X)$ be the quasi-Albanese map of the quasi-projective\footnote{Here again, the conclusion should hold in the quasi-K\"ahler case.} manifold $X$. Then $\alpha_X$ is dominant (hence surjective if it is proper) and has connected fibres if $X$ is special, in particular if $\overline{\kappa}(X)=0$. 
\end{lemma}

\begin{proof} It is sufficient to show that any Zariski-closed subset $V\subsetneq Alb(X)$ of $Alb(X)$  admits a dominant map onto a positive-dimensional quasi-projective $W$ of general type, that is such that $\overline{\kappa}(W)=dim(W)$.This follows from \cite{Kaw}, Theorem 27.
\end{proof}

\begin{remark}\label{notproper} The conclusion of Corollary \ref{cspec} may fail if $\alpha_X$ is not proper, even if $K_{\overline{X}}+D$ is trivial. See Example \ref{ex} below.
\end{remark}

Recall the Abelianity conjecture.

\begin{conjecture}\label{Ab conj} (\cite{Ca04}, Conjecture 7.1)Let $X$ be a compact K\"ahler manifold. If $X$ is special, then $\pi_1(X)$ is virtually abelian.
\end{conjecture}

This conjecture is shown to be true for the linear quotients of K\"ahler groups (\cite{Ca04}, Theorem 7.8). It is stated also in the quasi-K\"ahler case in \cite{Ca11}, Conjecture 12.10. However, non-compact special quasi-K\"ahler manifolds with torsionfree nilpotent, non-abelian, fundamental groups do exist, as first noticed in \cite{CDY}. See examples below. In the open case, we thus restate this as the `Nilpotency conjecture' (as done in \cite{CDY}):

\begin{conjecture}\label{Nilp conj} If $X$ is a special quasi-K\"ahler manifold, $\pi_1(X)$ is virtually nipotent. 

If $X$ is `slope-rationally connected', then $\pi_1(X)$ is virtually abelian.
\end{conjecture}

Recall (see \cite{Ca16}) that $X$ is `slope rationally connected' if quasi-projective, and if $h^0(\overline{X}, (\otimes^m\Omega^1_{\overline{X}}Log(D))\otimes H)=0$, for any ample line bundle $H$ on $\overline{X}$, provided $m\geq m(\overline{X},D,H)$. When $X$ is compact (i.e: $D=0$), this means that $X$ is rationally connected.

\section{Remarks and examples in the non-proper case}

\begin{example}\label{ex} Let $B$ be a connected complex manifold, $L$ a line bundle on $B$, $B'$ its zero section, and $X:=L\setminus B'$. If $B$ is quasi-projective, so is $X$. Notice that $X\to B$ also appears as the complement of the zero and infinity sections of $\mathbb P(L\oplus \mathcal{O}_B)\to B$.

\medskip

  The Gysin exact sequence for $L,X$ reads (\cite{MS}, Theorem 12.2, with $\mathbb Z$ coefficients), the fourth map being multiplication by $c_1(L)$:

$$0\to H^1(B)\to H^1(X)\to H^0(B)\to H^2(B)\to H^2(X)\to H^1(B)$$

Hence $H^1(X,\mathbb Z)=H^1(B,\mathbb Z)$ if $c_1(L)\neq 0$, moreover $\pi_1(X)$ is then a central extension of $\pi_1(B)$ by $\mathbb Z$, hence torsionfree nilpotent if so is $\pi_1(B)$. If $\pi_1(B)$ is torsionfree abelian, the nilpotency class of $\pi_1(X)$ is $2$. This happens if $B$ is, for example, an elliptic curve.

If the Log-canonical bundle of $B$ is moreover trivial, then so is the Log-canonical bundle of $X$, by a simple computation.

\medskip

Let us first take for $B$ an elliptic curve, and $L$ of nonzero degree $d$. Then $alb_X:X\to B:=Alb(X)$ is here not proper. The Logarithmic $1$-forms on $X$ then all come from the regular $1$-forms on $B$. Finally, $\pi_1(X)$ is an extension of $\pi_1(A)$ by $\pi_1(\mathbb C^*)=\mathbb Z$, hence torsionfree nilpotent (by \cite{Ca}, Appendix, Lemme A.1). It is a Heisenberg group, which is not K\"ahler (\cite{Ca}), but `quasi-K\"ahler'.

If we chose $B=\mathbb C^*\times \mathbb C^*$, with $L$ of any bidegree $(a,b)$ on its compactification $P:=\mathbb P_1\times \mathbb P_1$, $\pi_1(X)$ is abelian of rank $3$, since the natural restriction map $H^2(P,\Z)\to H^2(B)$ vanishes. In order to construct a quasi-projective $\C^*$-bundle over $\C^*\times \C^*$, one needs to change the algebraic structure, and chose as a good compactification the projective bundle $\pi:\P(E)\to C$, where $E$ is the non-trivial extension of two trivial line bundles over the elliptic curve $C$. Then $\P(E)\setminus B'$ is analytically, not algebraically, isomorphic to $B:=\C^*\times \C^*$. We obtain a $\C^*$-bundle with non-abelian fundamental group over $B$ by lifting one from $C$ as in the previous example.

After some versions of the current paper, the authors learned that the above examples were also considered in \cite[5.45]{HZ}.

\end{example}

\begin{remark} The preceding examples when $c_1(L)\neq 0$ show that two other properties of fibrations $f:X\to Y$ between compact K\"ahler manifolds do not extend to the case of non-proper surjective maps between quasi-K\"ahler manifolds, even if $Y$ is a projective curve and the map $f$ is a $\mathbb C^*$-bundle:

1. Even if the Log-canonical bundle is trivial (which is the case of $X=L\setminus B'$ in the previous example), $\pi_1(X)$ need not be virtually abelian. By contrast, recall that $\pi_1(X)$ is virtually abelian when $X$ is a compact K\"ahler manifold with $c_1(X)=0$. 

2. If $f:X\to Y$ is a `neat' fibration\footnote{`neatness' is achieved by suitably blowing up $X$ and $B$.} between compact K\"ahler manifolds, such that $\pi_1(X)$ is torsionfree nilpotent, of nilpotency class $\nu(\pi_1(X))$, then \cite{Ca-Ab}, Corollaire 5.2.1 shows that $$\nu(\pi_1(X))=max\{\nu(\pi_1(X_y)_X),\nu(\pi_1(B^{orb}))\}$$ if $\pi_1(X_y)_X$ is the image of $\pi_1(X_y)$ in $\pi_1(X)$, for $X_y$ any smooth fibre of $f$. Here $B^{orb}$ is the orbifold base of $f$. In particular: $\pi_1(X)$ is virtually abelian if so are $\pi_1(Y)$ and $\pi_1(X_y)_X$.

For non-K\"ahler nilmanifolds, $\nu(\pi_1(X))$ is the sum, not the $Max$ of the other two terms.
\end{remark} 

This shows that, at least for projective manifolds, one cannot construct torsionfree nilpotent K\"ahler groups of higher nilpotency class by constructing suitable fibrations. The only examples presently  known seem to be of nilpotency class $2$.

\begin{question}\label{q nilp} Let $G$ be a torsionfree nilpotent K\"ahler group. Can its nilpotency class $\nu(G)$ be at least 3? Can $\nu(G)$ be arbitrarily large?
Same questions for the quasi-K\"ahler groups.
\end{question}

 By contrast, the implication that $\pi_1(X)$ is virtually nilpotent if it is polycyclic remains true for $X$ quasi-K\"ahler (\cite{AN}, corollary 4.10).  It would be interesting to know if this remains true for solvable fundamental groups of quasi-K\"ahler manifolds, as Delzant theorem (\cite{De}) asserts that it is the case for K\"ahler groups. 

\begin{remark}\label{counter} The conclusions about the nilpotent cover and the universal cover being holomorphically convex commonly fail, as the following easy examples show. In the non-compact case, the conjecture of Shafarevich can thus extend only under restrictive conditions.

Let $M$ be an affine, or projective, smooth variety of dimension $n\geq 2$, and set $X:=M\setminus N$, where $N$ is a nonempty closed analytic subset of codimension at least $2$. Then $X$ is not holomorphically convex. If the universal cover $\widetilde{M}$ of $M$ is Stein, the universal cover $\widetilde{X}\subset \widetilde{M}$ of $X$ is not holomorphically convex. This happens by taking $M=\Bbb C^n, \Bbb P_n$, or an Abelian variety.

A more interesting example appears in \cite{W}, in which ones removes a suitable $N$ of codimension $2$ from the $4$-dimensional affine quadric $M$, the complement $X$ being diffeomorphic to $\Bbb C^4$.

\end{remark}

\section{Appendix: Weight-$1$ MHS}

The constructions made here are without doubt well-known, but we were unable to find a reference for them in the non-compact case.

We start with two holomorphic maps between irreducible quasi-projective varieties: $f:V\to \widetilde{X},g:\widetilde{X}\to Y$, where $\widetilde{X}$ is smooth, $Y$ normal, and $g$ is surjective and proper, with connected fibres. For $g:\widetilde{X}\to Y$, and the proof of Theorem \ref{qp normal} we may use $d:\widetilde{X}\to X$, the desingularisation of a normal quasi-projective $X$.

The maps $f,g$ induce $f^*:H^1(\widetilde{X})\to H^1(V)$ and $g^*:H^1(Y)\to H^1(\widetilde{X})$, $g^*$ is injective.

We shall define the Albanese map of $Y$ and describe some of its standard properties.

Let $H^1(Y):=Image(g^*)$. This is a sub-MHS of $H^1(\widetilde{X})$ defined over $\Z$. It thus defines\footnote{The construction for $M\subset H^1(\widetilde{X})$ is the same as for the quasi-Albanese variety: take the quotient $A_M:=(F^1(M))^*/M_{\Z}$, $F^{\bullet}$ being the Hodge filtration.} a semi-Abelian variety $A_Y$, quotient of $A:=Alb(\widetilde{X})$. Denote by $q_Y:A\to A_Y$ the quotient map, and by $\alpha:\widetilde{X}\to A:=Alb(\widetilde{X})$ the Albanese map.

\begin{proposition}\label{prop alb} If $V$ is either compact, or smooth, the following properties are equivalent:

1. $(g\circ f)^*:H^1(Y)\to H^1(V)$ vanishes.

2. Each Logarithmic one-form $u$ on $\widetilde{X}$ which belongs to $g^*(H^1(Y))$ vanishes on the smooth points of $V$.

3. $q_Y\circ f:V\to Alb_Y$ maps $V$ to a point. 
\end{proposition}

\begin{proof} That $2$ implies $3$ is clear, since $A_Y$ is obtained by integrating these $1$-forms along paths in $\widetilde{X}$. That $3$ implies $1$ follows from the equality $H^1(A_Y)=H^1(Y)$, true by the construction of $A_Y$. The last implication $1\Longrightarrow 2$ is clear if $V$ is smooth, since the logarithmic $1$-forms on $\widetilde{X}$ pulled back by $f$ on $V$ are elements of $H^1(V)$, and vanish cohomologically only if they vanish identically. If $V$ is projective (but not necessarily smooth), the assertion again follows from \cite{D}, 8.2.7, which says that $Ker(f^*)=Ker (g\circ v)^*:H^1(\widetilde{X})\to H^1(\widetilde{V})$, if $v:\widetilde{V}\to V$ is a desingularisation. One can thus assume $V$ to be smooth and apply the preceding argument.
\end{proof}

\begin{corollary}\label{cor alb} If $g:\widetilde{X}\to Y$ is proper, surjective, with $Y$ normal and $\widetilde{X}$ smooth, there is a unique map $\alpha_Y:Y\to Alb_Y$ such that  $\alpha_Y\circ g=q_Y\circ g:\widetilde{X}\to Alb_Y$. This map is called the Albanese map of $Y$. It does not depend on the choice of $g,\widetilde{X}$, and enjoys the usual universal property of Albanese maps.
\end{corollary}

\begin{proof} The irreducible components of the fibres of $g$ are compact, by assumption, and satisfy the first property of Proposition \ref{prop alb}. They are thus mapped to points by $q_Y\circ \alpha$, as this Proposition shows, which implies the claim. We do not use the two other statements, and do not prove them (the proof is standard). 
\end{proof}

\begin{corollary}\label{alb^*} For each component $G$ of any fibre of $g$, denote by $f_G:G\to \widetilde{X}$ the inclusion map, and $K_G:=Ker(f_G^*:H^1(\widetilde{X})\to H^1(G))$. Then $g^*(H^1(Y))=\cap_{G}K_G\subset H^1(\widetilde{X})$.
\end{corollary}

\begin{proof} The inclusion $g^*(H^1(Y))\subset\cap_{G}K_G$ is clear. In the opposite direction: $H:=\cap_{G}K_G$ is a sub MHS of $H^1(\widetilde{X})$ defined over $\Z$, and we thus have an associated quotient semi-Abelian variety $A_Y^*$ and a quotient map $q_g:A\to A_Y^*$. By construction, all $G's$ are mapped to points of $A_Y^*$ by $q_g\circ \alpha\circ f_G:G\to A_Y^*$. We thus get a unique $\alpha_Y^*:Y\to A_Y^*$ such that $\alpha^*_Y\circ g=q_g\circ \alpha:\widetilde{X}\to A_Y^*$. Since $q_g^*(H^1(A_Y^*))=H$ contains $g^*(H^1(Y))=\alpha^*(q_Y^*(H^1(A_Y)))$, we have a natural quotient $q':A_Y^*\to A_Y$ such that $q'\circ \alpha_Y^*=\alpha_Y:Y\to A_Y$. The universal property of $\alpha_Y:Y\to A_Y^*$ implies that $q'$ is isomorphic. Since we did not show this universal property, let us prove directly that $H=g^*(H^1(Y))$: we indeed have $(q_g\circ \alpha)^*(H^1(A_Y^*))=\alpha^*(q_g^*(H^1(A_Y^*)))=H$.

On the other hand, since $q_g\circ \alpha=q_Y^*\circ g$, $H$ is also equal to:

 $g^*((q_Y^*)^*(H^1(A_Y^*)))\subset g^*(H^1(Y))$. We thus get $H=g^*(H^1(H))$, and $A_Y=A_Y^*$.
\end{proof}

\begin{remark} The Albanese map $q_Y:Y\to A_Y$ defined here coincides for $s=1$ with the higher Albanese maps as defined by R. Hain. The Albanese map $q_Y^*:Y\to A_Y^*$ coincides with the one defined in \cite{CDY}, remark 1.2. We thus showed that all three coincide, and in particular that $H^1(A_Y^*)=H^1(Y)$, which depends on MHS considerations (at least using the proof given here).
\end{remark}

\bibliography{nilp}
\bibliographystyle{amsalpha}
\end{document}